\documentclass{amsart}
\usepackage{verbatim,amsfonts,color,mathrsfs,stmaryrd, bbold}
\usepackage[all]{xy}
\usepackage{amssymb,amsmath,amscd,epsf,xypic,amsxtra}

\theoremstyle{plain}
\newtheorem{Thm}{Theorem}
\newtheorem{Cor}{Corollary}

\newtheorem{Lem}{Lemma}
\newtheorem{Prop}{Proposition}

\theoremstyle{definition}
\newtheorem{Def}{Definition}

\newtheorem{Eg}{Example}

\theoremstyle{remark}

\define\wh{\widehat}
\define\wt{\widetilde}
\def\R{\mathbb R} 
\def\A{\mathscr A} 
\def\C{\mathscr C}

\begin{document}

 
\title{Distribution of approximants and geodesic flows}

\author{Albert M. Fisher} 
\address{Albert M. Fisher,
Dept Mat IME-USP,
Caixa Postal 66281,
CEP 05315-970
S\~ao Paulo, Brazil}
\urladdr{http://ime.usp.br/$\sim$afisher}
\email{afisher@ime.usp.br}

\author{Thomas A. Schmidt}
\address{Oregon State University\\Corvallis, OR 97331}
\email{toms@math.orst.edu}
\keywords{geodesic flow,  continued fractions}
\subjclass[2010]{37E05, 11K50, 30B70}
\date{27 July 2012}

\begin{abstract}  We give a new proof of Moeckel's result that for any finite index subgroup of the modular group, 
almost every real number has its regular continued fraction approximants equidistributed into the cusps of the subgroup according to the weighted cusp widths.    Our proof uses a skew product over a cross-section for the geodesic flow on the modular surface.   Our techniques show that the same result  holds true for approximants found by Nakada's $\alpha$-continued fractions,  and also that the analogous result holds for approximants that are algebraic numbers given by any of Rosen's $\lambda$-continued fractions, related to the infinite family of  Hecke triangle Fuchsian groups.
\end{abstract}

\maketitle

\section{Introduction}

 For each nonzero real number, the regular continued fraction algorithm produces a sequence of rational numbers that are best approximations in the appropriate sense.     It is natural to ask about the number theoretic properties of the numerators and denominators in each such sequence.   In particular, as we learned from Sheingorn \cite{Sheingorn93},  going back to the 1940s there has been a study of the distribution into congruence classes of these numerators and denominators.   The situation was beautifully resolved in 1982 by Moeckel \cite{Moeckel82}, who used the celebrated connection between regular continued fractions and geodesics on the hyperbolic modular surface.   

A fascinating aspect of the problem is that it does not directly involve observables of the dynamical system defined by the {\em Gauss map}, that is by the interval map related to the regular continued fractions.    Furthermore, the vocabulary of the result indicates some of its depth:  (1) any finite index subgroup $H$ of $\text{SL}(2, \mathbb Z)$ uniformizes the hyperbolic {\em surface} $\bar H \backslash \mathbb H$ where $\mathbb H$ is the Poincar\'e upper half-plane and $\bar H$ denotes the image of $H$ in  $\Gamma := \text{PSL}(2, \mathbb Z) = \text{SL}(2, \mathbb Z)/\pm I$; (2) a {\em cusp} $\kappa$ of $H$ is an $\bar H$-orbit in $\mathbb Q \cup \{\infty\}$; (3) the {\em width} of a cusp is the index $w(\kappa)$ of the $\bar H$-stabilizer of a point in this $\bar H$-orbit inside the point's $\bar \Gamma$-stabilizer.   

Combined with a later result  of Nakanishi \cite{Nakanishi89}, the Moeckel result, that we reprove here --- and then generalize --- is the following.   

\begin{Thm} [Moeckel, Nakanishi]\label{t:completeMoeckel}    Suppose that $H$ is a finite index subgroup $H \subset \emph{SL}(2, \mathbb Z)$.    Then,  for almost every real $x$,  the regular continued fraction approximants of $x$ are distributed in the cusps of $H$ according to the relative cusp widths.  That is, for each cusp $\kappa$ and almost all $x$, 
\begin{equation}\label{e:moeckelLim}
 \lim_{N\to \infty}\; \dfrac{\#\{0\le n \le N\,|\, p_n/q_n\in \kappa\}}{N} = w(\kappa)/[\Gamma: \bar H]\,.
 \end{equation}
\end{Thm} 

Since for each positive integer $m$,  the principal congruence subgroup $\Gamma(m)$ defined as the kernel of the homomorphism to the projective matrix group over the finite ring $\mathbb Z/m \mathbb Z$ is a normal subgroup,  the result states that for almost all $x$, the approximants $p_k/q_k$ are equidistributed in the cusps of $\Gamma(m)$.   In particular,  when $m=2$ the result is that for almost all $x$,  its approximants have limiting frequency of 1/3 for each of the type {\tt odd/odd}, {\tt odd/even} and {\tt even/odd}.

Moeckel relied on the ergodicity of the geodesic flow on the unit tangent bundle of any finite cover of the modular surface.       
A few years thereafter,   Jager and Liardet \cite{JagerLiadet88} gave a new proof (in the main setting of the principal congruence subgroups) that was in a sense purely ergodic theoretic;  in particular,  they used a  skew product over the Gauss map  to avoid the use of the geodesic flow.      The logical structure of both proofs is of two steps:  an ergodicity argument and a counting argument.   Whereas Moeckel's counting argument seems  slightly artificial --- indeed, it  forced a restriction to a particular subset of the finite covers of the modular surface,  a restriction later removed by Nakanishi but only by the addition of a second counting argument in the complementary case --- Jager and Liardet's counting argument is straightforward.    On the other hand,  the ergodicity argument in Moeckel's approach is in some sense immediate,  but Jager and Liardet must labor in the corresponding step.     

We give a combined proof that unites the strengths of the two approaches.    By using a skew product based over a cross-section of the geodesic flow on the unit tangent bundle of the modular surface,   ergodicity is immediate and counting is straightforward.   Furthermore,  our proof needs no adjustment to treat the setting of any finite cover of the modular surface.      The cross-section of the geodesic flow on the unit tangent bundle of the modular surface that we use is that found by Arnoux \cite{Arnoux94}; this  has the Gauss map as a factor transformation.   To be more exact,  the first return map to Arnoux's cross-section gives a double covering of the natural extension of the Gauss map.           Our approach is quite flexible,  as we show by not only recovering the Moeckel-Nakanishi results for the approximants generated by the regular continued fractions,  but also extending these results to Nakada's $\alpha$-continued fractions \cite{Nakada81},   and to Rosen's $\lambda_m$-continued fractions \cite{Rosen54}.   These latter are related not to merely the modular group, but to each of an infinite family of triangle Fuchsian groups,  known as the Hecke groups;  the Rosen fractions give approximation by quotients of certain algebraic integers.\\

The study of the intertwined nature of continued fractions and the geodesic flow on the modular surface has a rich history.    One of the first significant steps was in 1924, when  E.~Artin \cite{Artin24} gave a coding of geodesics in terms of the regular continued fraction expansions of the real endpoints of their lifts.     A decade later, Hedlund \cite{Hedlund35}  used this to show the ergodicity of the geodesic flow on the modular surface.   A few years thereafter,   E.~Hopf proved ergodicity for the geodesic flow on any hyperbolic surface of finite volume, see his reprisal \cite{Hopf71}.        Series~\cite{Series85} (inspired by Moeckel's result) and Adler and Flatto, see especially their  \cite{AdlerFlatto91}, give cross-sections for the geodesic flow on the modular surface such that the regular continued fraction map is given as a factor.    Quite recently,   
D.~Mayer and co-authors ~\cite{MayerStroemberg08}, ~\cite{MayerMuehlenbruch10}  have given explicit cross-sections for geodesic flow on the surfaces uniformized by the Hecke triangle groups such that a certain variant of the Rosen continued fractions \cite{Rosen54}  occurs as a factor.     (The second-named author first learned of such possibilities from A.~Haas, see the related treatment of another variant of the Rosen fractions in \cite{GroechenigHaas96}.)   For much of this history,  further motivation,  and also the work of S.~Katok and co-authors using various means to code geodesics on the modular surface,  see  \cite{KatokUgarcovici07}.  As stated above, for the regular continued fractions we use a cross-section given by Arnoux \cite{Arnoux94}; for the other continued fractions we study, we use cross-sections found by Arnoux and Schmidt \cite{ArnouxSchmidt12} with a method directly inspired by \cite{Arnoux94}.   (Arnoux's approach to constructing cross-sections is  more algebraic in flavor than the others we have mentioned.)\\   

Our approach does have the disadvantage of presenting consecutive approximants  as first or second columns alternately (see  Lemma~\ref{l:prodMatricesModH}), but this is natural when directly following the Gauss map.    We address this by twisting from a conjugate of the subgroup of interest to the subgroup itself, see  the end of the proof of Corollary~\ref{c:switchWithIota}.\\

After this work was completed,  we learned from Pollicott's recent survey 
 article \cite{Pollicott11} that C.-H. Chang and D. Mayer \cite{ChangMayer00} had previously used the skew product approach to lift a cross-section for geodesic flow on the unit tangent bundle of the modular surface to give cross-sections for finite index covers;  this allows them to study dynamical zeta functions for these covers.

\section{Background} 

\subsection{Regular continued fractions}  

The regular continued fractions map, or the Gauss interval map, on 
$\mathbb{I} : = [0, 1)$ is   
\[
T(x) := \frac{1}{x}   -  \left \lfloor\,  \frac{1}{x} \right \rfloor\, \mbox{for}\ x \neq 0\,; \ T(0) :=0 \,.
\]
For $x \in \mathbb{I}$, put
$d(x) := \left\lfloor  1/x  \right\rfloor\,$, 
with $d(0) = \infty$,  and set $ d_n = d_{n}(x) := d(T^{n-1}(x))$.
This yields the  regular continued fraction  expansion of
$x\in \mathbb R$\,:
\[
x = d_0+ \dfrac{1}{d_1 + \dfrac{1}{d_2 +
\cdots}} =: [\, d_0;\, d_1, d_2,\, \ldots\, ]
\,,
\]
where $d_0\in\mathbb Z$ is such that $x-d_0\in \mathbb{I}$.
(Standard convergence arguments justify equality of $x$ and its
expansion.)  
 The {\em approximants} $p_n/q_n$
of $x\in \mathbb I$ are given by
\[
\begin{pmatrix}
p_{-1}  &  p_0  \\
q_{-1}  &  q_0
\end{pmatrix}
\, = \,
\begin{pmatrix}
1  &  0  \\
0  &  1
\end{pmatrix}
\]
and
\begin{equation}\label{e:convMatrixForm}
\begin{pmatrix}
p_{n-1}  &  p_{n}  \\
q_{n-1}  &  q_{n}
\end{pmatrix}
\, = \,
\begin{pmatrix}
0 &  1 \\
1  &    d_{1}
\end{pmatrix}
\begin{pmatrix}
0 & 1  \\
1  &  d_{2}
\end{pmatrix}
\cdots
\begin{pmatrix}
0 &  1  \\
1  &  d_{n}
\end{pmatrix}
\end{equation}
for $n \ge 1$.   Of course,  this is equivalent to 
$p_n/q_n = [d_1, d_2, \dots, d_n]$.   
From the definition it is immediate that $\left|
p_{n-1}q_{n} \, - \, q_{n-1}p_n \right| =1$ and in particular that 
the rational numbers $p_n/q_n$ are all in reduced form. \\

The standard number theoretic planar map associated to continued fractions is 
\[
\mathcal{T}(x,y) := \bigg(T(x), \frac{1}{d(x)+  y}\bigg)\,, \quad (x, y) \in \Omega\,
\]
with   $\Omega =  \mathbb{I} \times [0,1]$.     
The measure $\mu$ on $\mathbb R^2$ given by 
\begin{equation}\label{e:mu}
d \mu = (1 + xy)^{-2} dx dy
\end{equation}
 is invariant for $\mathcal T$.  
Indeed,   Nakada, Ito, Tanaka \cite{NakadaItoTanaka77}   showed that this two dimensional system gives a natural extension for the map $T$.   In particular, 
 the marginal measure obtained by integrating $\mu$ over the fibers~$\{x\} \times \{y \mid (x,y) \in \Omega\}$ gives the invariant measure $\nu$ with $d \nu = (1+x)^{-1} dx$.   The probability measure on $\mathbb I$ obtained by division by $\log 2$ is known as Gauss measure.

 \subsection{Geodesic flow}

Some of the following can be found in  \cite{Manning91}.\\  

    Using the M\"obius action of $\text{SL}_2(\mathbb R)$ on the Poincar\'e upper-half plane $\mathbb H$,  one can identify a matrix with the image of $z=i$ under the matrix.   This results in  $\text{SL}_2(\mathbb R)/ \text{SO}_2(\mathbb R)$ being identified with the Poincar\'e upper half-plane $\mathbb H$.   Similarly,   $G := \text{PSL}_2(\mathbb R) = \text{SL}_2(\mathbb R)/\pm I$ can be identified with the {\em unit tangent bundle} of $\mathbb H$.      Recall that a Fuchsian group $\Gamma$  is any discrete subgroup of $G$, equivalently $\Gamma$ acts properly discontinuously on $\mathbb H$ and the quotient is a hyperbolic surface (to be precise, possibly an orbifold).   The unit tangent bundle of this quotient can be identified with $\Gamma\backslash G$.    
    
The geodesic flow in our setting acts on the surface's unit tangent bundle:  Given a time $t$ and a unit tangent vector $v$,   since the unit tangent vector uniquely determines a geodesic passing through the vector's base point,  we can follow that geodesic for arclength $t$ in the direction of $v$;   the unit vector that is tangent to the geodesic at the end point of the geodesic arc is the image,  $g_t(v)$, under the geodesic flow.    The hyperbolic metric on $\mathbb H$ corresponds to an element of arclength satisfying $ds^2 = (dx^2 + dy^2)/y^2$ with coordinates $z = x + i y$.   In particular,  for $t>0$,   the points $z = i$ and $w = e^t i$ are of distance $t$ apart.  Since $G$ acts on the left by  isometries on $\mathbb H$,   the geodesic flow on its unit tangent bundle is given by sending $A\in G$ to $A E_t$,   where $E_t = \begin{pmatrix} e^{t/2}&0\\0&e^{-t/2}\end{pmatrix}$.     Similarly,  on $\Gamma\backslash G$   one sends the class represented by $A$ to that represented by $A E_t\,$; we sometimes denote this by  $g_t(\, [A]_{\Gamma}\,) =     [A E_t]_{\Gamma}$.   

There is a natural measure on the unit tangent bundle $T^1 \mathbb H$:   {\em Liouville measure} is given  as the product of the hyperbolic area measure on $\mathbb H$ with the length measure on the circle of unit vectors at any point.   This measure is (left- and right-) $\text{SL}_2(\mathbb R)$-invariant,  and  thus gives Haar measure on $G$.   In particular, this measure is invariant for the geodesic flow.       When discussing measure theory,  we  always consider the Borel $\sigma$-algebra or its $\mu$-completion.

Let $(X, \mathscr B, \mu)$ be a measure space and $\Phi_t$ a measure preserving flow on $X$,  that is $\Phi:  X \times \mathbb R \to X$ is a measurable function such that for $\Phi_t(x) = \Phi(x,t)$, $\Phi_{s+t} = \Phi_s \circ \Phi_t$.   Then $\Sigma \subset X$ is a  {\em measurable cross-section} for  the flow $\Phi_t$ if:     (1)  the flow orbit of almost every point meets $\Sigma$; (2) for almost every $x\in X$ the set of times $t$ such that $\Phi_t(x) \in \Sigma$ is a discrete subset of $\mathbb R$; (3)   for every  $\tau> 0$, the {\em flow box}  $A_{[0,\tau]}\equiv \{ \Phi_t(A)\, \mid \, A \in \Sigma,  t \in [0, \tau]\,\}$ is $\mu$-measurable.
This last defines the induced $\sigma$-algebra  $\mathscr B_\Sigma$ on $\Sigma$.

The {\em return-time function} $r= r_\Sigma$ is $r(x) =\inf\{t>0:\Phi_t(x)\in\Sigma\}$ and the {\em return map} $R:\Sigma\to\Sigma$ is defined by $R(x)= \Phi_{r(x)}(x)$. 
The {\em induced measure} $\mu_\Sigma$ on $\Sigma$ is defined from flow boxes: one sets 
$\mu_\Sigma(A)= \frac{1}{\tau} \mu(A_{[0,\tau]})$ for (all)  $0<\tau<\inf_{x\in A}\{r(x)\}$.
The flow $(X, \mathscr B, \mu)$ is then naturally isomorphic to the {\em special flow}
built over the return map 
 $(\Sigma,  \mu_\Sigma, R)$ with return time $r$, as follows: define $\wh \Sigma= \{(x,t):\, 0\leq t\leq r(x)\}/\sim$, where $\sim$ is the equivalence relation $(x,r(x))\sim (R(x), 0)$,
equipped with the measure $\wh\mu$, the product  of $\mu_\Sigma$ on $\Sigma$ with  Lebesgue measure on $\R$; the flow $\Phi^R$ on $\wh\Sigma$  is defined by $\Phi^R_t:\, (x,s)\mapsto (x, s+t)$; this preserves $\wh\mu$.

\medskip

\noindent
{\bf Convention}:  In all that follows,  we will write {\em cross-section} to denote measurable cross-section.
 \smallskip

A flow $\Phi_t$ is {\em ergodic} if for any invariant set either it or its complement is of measure zero.   Of fundamental importance for us is Hopf's result that the geodesic flow is ergodic on the unit tangent bundle of any finite volume hyperbolic surface, see his reprisal in \cite{Hopf71}.    A flow is {\em recurrent} if the $\Phi$-orbit of almost every point meets any positive measure set infinitely often.  By the Poincar\'e Recurrence Theorem, an ergodic flow on a finite measure space is recurrent.    Given a cross-section, a return-time transformation  $(\Sigma, \mathscr B_\Sigma, \mu_\Sigma, R_\Sigma)$ is ergodic and recurrent if and only if the flow is. A {\em semiconjugacy}, {\em homomorphism} or {\em factor map} of two 
measure-preserving flows or transformations is an a.s. onto map which preserves the measures and the dynamics; a {\em conjugacy} or {\em isomorphism} is in addition invertible, with the inverse map also a homomorphism. 

The following summarizes some of the above from a perspective of central importance to us. 
\begin{Lem}\label{l:returnPair}   Let $G$ denote the group $\emph{PSL}(2, \mathbb R)$.   
Suppose that $\Gamma$ is a Fuchsian group of cofinite volume and that $\mathscr A \subset G$ projects injectively to give a cross-section $\Sigma$ for the geodesic flow on the surface uniformized by $\Gamma$.      Then for almost every $A \in \mathscr A$ there exists a unique pair of first return time $t = t_A$ and $M \in \Gamma$ such that $M A E_t \in \mathscr A$.   
 \end{Lem}

We shall also need the following about lifts of cross-sections:
\begin{Lem}\label{l:lift}
  Let  $(\wt X, \wt{\mathscr B},\wt \mu, \wt \Phi_t)$ and $(X, \mathscr B, \mu, \Phi_t)$ be two measure-preserving flows such that the second is a factor of the first, i.e. there exists $\pi: \wt X\to X$ measure-preserving and onto such that 
$\pi\circ \wt \Phi_t= \Phi_t\circ \pi$. Let $(\Sigma,  \mu_R, R)$ be a 
cross-section and return map for $ \Phi_t$. Then 
$\wt \Sigma\equiv \pi^{-1}(\Sigma)$ is a cross-section for  $ \wt\Phi_t$, 
with return function $\wt r= r\circ \pi$, return map $\wt R= \wt \Phi_{\wt r(\wt x)}(\wt x)$ and measure $\wt \mu$ on $\wt\Sigma$ the induced measure;  the two return transformations are semiconjugate by way of the restriction of $\pi$ to $\wt \Sigma$. 
\end{Lem}

The proof follows from the definitions. We call $(\wt\Sigma,  \wt{\mathscr B}_\Sigma, \wt\mu, \wt R)$ the {\em lifted cross-section}.

\subsection{Arnoux cross-sections}     Inspired by Veech's ``zippered rectangles''  \cite{Veech84} as a means to study the geodesic flow on Teichm\"uller spaces,   Arnoux \cite{Arnoux94} gave a coding of geodesic flow on unit tangent bundle of the modular surface $\text{PSL}(2, \mathbb Z)\backslash G$ by way of ``boxes'' described in terms of two types of subsets of $\text{SL}_2(\mathbb R)$.   
Given   $(x, y) \in \mathbb R^2$,  let 
\begin{align}\label{e:basicMatShapes}  A_{-1}(x,y) &= \begin{pmatrix} 1&y\\-x & 1 - x y   \end{pmatrix}\;\;\mbox{and} \notag \\
\\
A_{+1}(x,y) &= \begin{pmatrix} x & 1 - x y \\ -1 & y  \end{pmatrix}\,.
\notag
\end{align}
Let 
\begin{align}\label{e:zMap} 
\mathcal Z: \mathbb R^2 \setminus \{(x,y)\,\vert\, y = -1/x\}  &\to \mathbb R^2\\
               (x,y) &\mapsto (x, y/(1+ x y)\,)\,\notag
\end{align} 
and define the following subsets of $\text{SL}_2(\mathbb R)$ (recall that $\Omega$ is the domain of the planar natural extension of the regular continued fractions map):   
\[ \mathcal A_{ -1} = \{ A_{-1}(x,y) \,     \vert \, (x, y) \in \mathcal Z^{-1}(\Omega)\}\;\mbox{and} \;  \mathcal A_{+1} = \{ \; A_{+1}(x,y) \, \vert \, (x, y) \in\mathcal Z^{-1}(\Omega)\}\,,   \]
and finally let $\mathscr A = \mathcal A_{ -1}  \cup \mathcal A_{ +1}$, considered as elements of $G$. 
      
Arnoux \cite{Arnoux94}  shows the following. 
\begin{Thm}\label{t:arnouxCross} [Arnoux] The set $\mathscr A \subset G$ projects injectively to give a cross-section for the geodesic flow on the unit tangent bundle of the modular surface, thus on $\Gamma\backslash G = \emph{PSL}(2, \mathbb Z)\backslash \emph{PSL}(2, \mathbb R)$.      Furthermore,   the first return map to this cross-section by the geodesic flow is given by the projection of 
\[ 
A_{\sigma}(x,y)  \mapsto  M A_{\sigma}(x,y) E_t \,,  
\]
with $t = -2 \log x$ and 
\[ 
M =  \begin{cases} \begin{pmatrix}1&\lfloor 1/x \rfloor\\0&1 \end{pmatrix} &\text{if $\sigma = -1$;}\\
\\
                               \begin{pmatrix}1&0\\\lfloor 1/x \rfloor&1 \end{pmatrix}&\text{if $\sigma = +1$.}
                  \end{cases}
\]
Moreover,  the dynamical system defined by this first return map has the Gauss map as a factor via the projection to the first coordinate.  
\end{Thm}

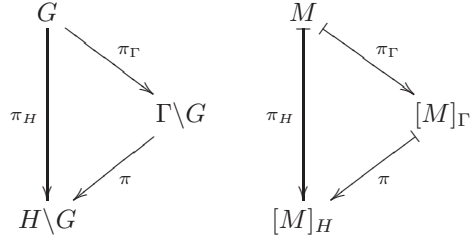
\begin{figure}
\centerline{
\xymatrix{
G \ar[dd]_{\pi_{H}} \ar[rd]^{\pi_{\Gamma}}  \\
&\Gamma\backslash G \ar[ld]^{\pi}\\
H\backslash G\\
}
\;\;\;\;\;\;\xymatrix{
M \ar@{|->}[dd]_{\pi_{H}} \ar@{|->}[rd]^{\pi_{\Gamma}}  \\
&[M]_{\Gamma}\ar@{|->}[ld]^{\pi}\\
[M]_{H}\\
}
}
%
\caption{Basic notation of covers: sets and element-wise maps.}
\label{f:projs}
\end{figure}
\subsection{Cusps, cosets and covers}   
An element of $G= \text{PSL}(2, \mathbb R)$ is parabolic  if it fixes a unique element of $\mathbb R\cup \{\infty\}$.   For the corresponding $A = \begin{pmatrix} a&b\\c&d\end{pmatrix} \in \text{SL}(2, \mathbb R)$, this is equivalent to $|\text{tr} A| = |a+d|$  being equal to $2$.     If $\Gamma$ is a Fuchsian subgroup of $G$, then the $\Gamma$-orbit of any of its parabolic fixed points consists solely of parabolic fixed points for $\Gamma$; we call such an orbit a {\em cusp} of $\Gamma$.  It is also  a standard fact that the $\Gamma$-stabilizer of any such parabolic fixed point is infinite cyclic.     By conjugating as necessary,  we can assume that the parabolic fixed point is at infinity, and thus the stabilizer is a cyclic group of translations;  the cusp of the group corresponds to a finite area end of the hyperbolic surface $\Gamma \backslash \mathbb H$, traditionally also called a {\em cusp} --- with our normalization, this geometric cusp  is the projection of a hyperbolic triangle with one vertex at infinity, where  the stabilizer acts so as to identify the two sides of the triangle.

  If $H \subset \Gamma$ is any finite index subgroup,  then any parabolic fixed point of $\Gamma$ is also a parabolic fixed point of $H$.  In particular,  each cusp of $\Gamma$ is partitioned into cusps of $H$.      The {\em width} (with respect to $\Gamma$)   of a cusp $\kappa$ of $H$  is the index $w(\kappa)$ of the $H$-stabilizer of any representative for the cusp in the $\Gamma$-stabilizer of this representative parabolic fixed point.    The geometric picture is that one can glue together $w(\kappa)$ copies of the triangle in $\mathbb H$ giving the cusp of the surface uniformized by $\Gamma$ to construct the triangle giving the chosen cusp for  the surface uniformized by $H$.

   We denote a right $H$-coset in $\Gamma$ as either $H \gamma$ or $\gamma \;\text{mod}\; H$, depending upon ease of typography.    

Now,  given a parabolic fixed point $p$ of $\Gamma$  and a coset $H\gamma$,  since $\gamma\cdot p$ is a parabolic fixed point (also) for $H$, it lies in some cusp $\kappa$ of $H$.   But, then so does its whole $H$-orbit; that is $H\gamma$ sends $p$ into $\kappa$.   There is thus an equivalence relation on the set of cosets $H\backslash \Gamma$,  characterized by sending $p$ to the same cusp.      Suppose that $\Gamma$ has exactly one cusp.  Since $\Gamma$ acts transitively on its single cusp (viewed as a $\Gamma$-orbit),  and $\Gamma$ is the union of the right cosets of $H$ in $\Gamma$,  we see that every $H$-cusp $\kappa$ has a nonempty corresponding equivalence class.   Let $w'(\kappa)$ be the number of elements in this class.     

Choose a fundamental domain,  $\mathscr F$,  for $\Gamma$ acting on $\mathbb H$ which reaches the boundary at $p$.    We can also choose $S \in \Gamma$ generating the $\Gamma$-stabilizer of $p$.  
Upon choosing coset representatives $\gamma_i$,   a fundamental domain for $H$ is given by taking the union of these $\gamma_i$ applied to $\mathscr F$.   Each $\gamma_i$ in the equivalence class of $\kappa$ conjugates  the generator $S$   to a generator of the $\Gamma$-stabilizer of their common image of $p$.   Indeed,  the $H$-stabilizer of this image point is generated by the product over the equivalence class of the $\gamma_i S \gamma_{i}^{-1}$.  Therefore,  we have the equality $w'(\kappa) = w(\kappa)$.

 We summarize the above in the following.

\begin{Lem}\label{l:cuspWidth}       
Suppose that $\Gamma$ is a Fuchsian group with a single cusp and that $H$ is a finite index subgroup of $\Gamma$.    Then upon choice of a parabolic fixed point of $\Gamma$,  there is an equivalence relation on the right cosets $H\backslash \Gamma$, corresponding to cusps of $H$.   The size of an equivalence class is the width of the corresponding cusp. 
\end{Lem}

We define a notation for the equivalence classes on the right cosets.
\begin{Def}\label{d:cuspClassDenote}       
 Given $\Gamma$ and $H$ as in Lemma~\ref{l:cuspWidth}, and a parabolic point $p$ for $\Gamma$, 
 let 
 \[ [ H \gamma]_{\cdot p}\]
 denote the equivalence class of the right coset $H \gamma$ under the relation defined by  equality of image of $p$, and 
 \[\{\kappa\}_{\cdot p}\]
 denote the partition block of $H\backslash \Gamma$ corresponding to $\kappa$ under this equivalence relation.  
 We refer to  the relation itself as the {\em $\cdot p$ relation} on $H\backslash \Gamma$.
\end{Def}

Note that we will occasionally slur over distinctions and refer to the various equivalence classes of $H\backslash \Gamma$ defined by a $\cdot p$ relation as the cusps of $H$.\\

We will have need to switch roles between a group $H$ and its conjugate by 
\begin{equation}\label{e:iota}
 \iota = \begin{pmatrix} 0&-1\\1&0\end{pmatrix} \in G\,.
 \end{equation}
Note that $\iota^{-1} = \iota$ in $G$.   

\begin{Lem}\label{l:congGroups}       
Suppose that $\Gamma$ is a Fuchsian group with a single cusp, with $\iota \in \Gamma$,  and that $H$ is a finite index subgroup of $\Gamma$.   If $\{\gamma_i\}_{1\le i \le n}$ is a full set of  coset representatives for $\iota H \iota\backslash \Gamma$, then the $\iota \gamma_i$ form a full set of coset representatives of $H\backslash \Gamma$.   Furthermore,  the map $p \mapsto \iota\cdot p$ places the $\iota H\iota$-cusps   in 1-to-1 correspondence with the $H$-cusps.   This correspondence preserves cusp widths.
\end{Lem}
\begin{proof} If $\Gamma$ is partitioned by the $\iota H\iota \gamma_i$,  then since $\iota \in \Gamma$ we have that 
$\Gamma = \iota \Gamma$ is partitioned by the $H\iota\gamma_i$.   Thus the $\iota\gamma_i$ do give  a full set of $H$-cosets.   

 If $p$ is a parabolic fixed point for $\Gamma$,  then left multiplication by $\iota$ sends $\iota H\iota\cdot p$ to $H(\iota\cdot p)$.
If $\iota H\iota \gamma_i\cdot p = \iota H\iota\gamma_j\cdot p$,  then by left multiplication by $\iota^{-1} = \iota$, we immediately have $H (\iota \gamma_i)\cdot p = H (\iota\gamma_j)\cdot p$.    Thus,  the equivalence relation defined by $p$ gives the same cusp widths for $H$ as for $\iota H\iota$.

\end{proof}

 For $H$ a subgroup of $G$, we also use the notation for the associated projections shown in Figure~\ref{f:projs}   --- in particular,  for ease of reading, we often use an expression of the form $[M]_{\Gamma}$ to denote $\pi_{\Gamma}(M)$.

\section{Lifting cross-section as skew product} 
 
  Given a cross-section for the geodesic flow on a hyperbolic surface and a finite degree cover of the surface, from the general facts of Lemma \ref{l:lift}
the cross-section can be lifted to give a cross-section for the geodesic flow on the  cover; as we see here, in this case  the resulting system can be nicely expressed as a skew product over the original cross-section.
\begin{Thm}\label{t:skew}    
Suppose that $\Gamma$ is a Fuchsian group of cofinite volume and that the projection map $\pi_{\Gamma}$ injectively maps a subset $\mathscr A \subset G$ 
to  a cross-section $\Sigma$ for the geodesic flow on $\Gamma\backslash G \cong T^1(\Gamma \backslash \mathbb H)$.       Let $H \subset \Gamma$ be any finite index subgroup of $\Gamma$, and  $H\backslash \Gamma$ the set of right cosets of $H$ in $\Gamma$.    Then the skew product transformation 
\[
\begin{aligned} 
\mathcal S: \mathscr A\,  \times\; &H\backslash \Gamma &\to \;\;\;\;\;\mathscr A \,  \times\; &H\backslash \Gamma  \;\;\;\;\;\;\text{defined by}\\
(\,A,\;& \gamma\,\emph{mod}\, H\,)  &\mapsto (\, MAE_t,  \, &\gamma M^{-1} \,\emph{mod}\, H\,)\,,
\end{aligned}
\]
where $M$  and $E_t$ are as in Lemma~\ref{l:returnPair},  
defines a dynamical system naturally isomorphic to the first return map on the canonical lifted  cross-section of the geodesic flow on $H\backslash G \cong T^1(H \backslash \mathbb H)$. 
The measures involved are these: on $\Gamma\backslash G$ and $H\backslash G$ we have Liouville measure, and their cross-sections $\Sigma,$ $\wt\Sigma$  carry the natural
 induced measures;    the measure on $\A$ is passed over via the bijection to $\Sigma$ and on the finite set $H \backslash \Gamma$ we take 
counting measure, with the skew product carrying the   product of these. 
\end{Thm}

See  Figure~\ref{f:liftCross} for diagrams related to this proof.    We use juxtaposition of sets of group elements to denote the set of all products of respective elements of these sets.
 
\begin{proof}   
 
 Choose a set  of right coset representatives $\mathscr C \subset \Gamma$ for $H\backslash \Gamma\,$.

 Consider  $A_0 \in \mathscr A$;   its projection $[A_0]_{\Gamma}\in \pi_\Gamma(\A)=\Sigma$ is sent to some $[A_1]_{\Gamma}$ by the first return map $R$ on $\Sigma$. We claim that $R$ is isomorphic to the map on $\A$ defined by 
$A_0\mapsto  M_0 A_0 E_0\,$. To check this, by Lemma \ref{l:returnPair}  
the geodesic flow on $\Gamma\backslash G$   takes $[A_0]_{\Gamma} $ to $[A_1]_{\Gamma}$, with $A_1 = M_0 A_0 E_0\,$, where as usual   
 $M_0 \in \Gamma$ and $E_0$ is the diagonal matrix realizing the flow for a time of $t_0\,$.
Next we consider  $\gamma_0 \in \mathscr C\,$ and $\gamma_0 A_0\in \C\A$; we claim that the return map $\wt R$ for the lifted cross-section $\wt\Sigma$ is isomorphic to the map on $\C\A$ defined by 
$\gamma_0 A_0\mapsto \gamma_1 A_1$ where $\gamma_1= \gamma_0M^{-1}$.
But
  since $\Gamma$ uniquely decomposes as $H \mathscr C\,$, there exists a unique pair $h \in H, \gamma_1 \in \mathscr C$ such that $h^{-1} \gamma_1 = \gamma_0 M_{0}^{-1} \in \Gamma\,$.  Equivalently,    $h \gamma_0 = \gamma_1 M_{0} $ and thus,
\[ h \gamma_0 A_0   E_0 = \gamma_1 M_0 A_0 E_0= \gamma_1 A_1\,.\]
That is,  $[\, \gamma_0 A_0\,]_{H}$ returns to $\pi_H(\,\mathscr C \mathscr A\,)$  at time $t_0\,$.  This is, moreover, 
the {\em first} return to this set:  suppose that there  are $h' \in H, \gamma \in \mathscr C, A \in \mathscr A$ and $t>0$ such that $h'  \gamma_0 A_0 E_t = \gamma A\,$.    The element in $\Gamma$ given by $M = \gamma^{-1} h'  \gamma_0$ gives $M A_0 E_t = A\,$ and thus $t$ is also a $\Phi_{\Gamma}$-return time for $[A_0\,]_{\Gamma}\,$,  showing that $t \ge t_0\,$.    Therefore,   $t_0$ is indeed the first return time for $[\gamma_0 A_0]_{H}$.

 Since the 
  map sending  $(A, H\gamma)$ to $[ \gamma A]_H$ is clearly surjective from $\mathscr A \times  H \backslash \Gamma$ to $\pi_H( \,\mathscr C \mathscr A\,)$, it remains to show it is also injective.    Now suppose that $A, A' \in \mathscr A$ and $\gamma, \gamma' \in \Gamma$ such that $[\gamma A]_{H} = [\gamma' A']_{H}$.   It follows that also $[\gamma A]_{\Gamma} = [\gamma' A']_{\Gamma}$ and hence $\pi_{\Gamma}(A) = \pi_{\Gamma}(A')\,$;  by the injectivity of $\pi_{\Gamma}$ on $\mathscr A$, we have $A = A'$.   But, $[\gamma A]_{H} = [\gamma' A']_{H}$ implies the existence of $h \in H$ such that $h \gamma A = \gamma' A'$;  since $A = A'$ we find that $\gamma$ and $\gamma'$ are in the same right $H$-coset of $\Gamma$.

We have 
verified that our map is a bijection, which conjugates the skew product transformation to 
the return map of $\wt\Sigma$. The measures on the cross-sections  correspond as well,
since the Liouville measure on  $H\backslash G$ may be viewed as the restriction, or projection, of Haar measure on   $G$, while the induced measure on the cross-section 
$\wt\Sigma$ is the product of Haar measure on $G$ restricted to $\mathscr A$ times counting measure on $H \backslash \Gamma$ (since a flow box for the lifted cross-section is a finite disjoint union of copies of the corresponding flow box for the section).

Finally, since the skew product is independent of choice of coset representatives, we find that the isomorphism is natural. 
This completes the proof.  
\end{proof}

\begin{figure}
\centerline{
\xymatrix{
\mathscr A \times  H \backslash \Gamma \ar[d] \ar[r]^{\mathcal S} &\mathscr A \times  H \backslash \Gamma  \ar[d] \\
\widetilde{\Sigma} \ar[d]^{\pi\;\;} \ar[r]^{\wt R\;\;} &\widetilde{\Sigma} \ar[d]^{\pi\;\;}\\
 \Sigma  \ar[r]^{R\;\;} &\Sigma\\ 
}
\;\;\;\;\;\;\xymatrix{
(\, A, \, \gamma \,\text{mod}\, H )  \ar@{|->}[d] \ar@{|->}[r]^<<<<<{\mathcal S} &(  MA E_t, \, \gamma M^{-1} \,\text{mod}\,  H  )   \ar@{|->}[d] \\
[\gamma A\,]_{H} \ar@{|->}[d]^{\pi\;\;} \ar@{|->}[r]^{\wt R\;\;} &[\, \gamma A E_t\,]_{H} \ar@{|->}[d]^{\pi\;\;}\\
[A\,]_{\Gamma}  \ar@{|->}[r]^{R\;\;} &[\, A E_t\,]_{\Gamma}   \\ 
}
}
%
\caption{Lifting the cross-section $\Sigma = \pi_{\Gamma}(\,\mathscr A\,)$.  Diagrams of sets and element-wise maps. }
\label{f:liftCross}
\end{figure}
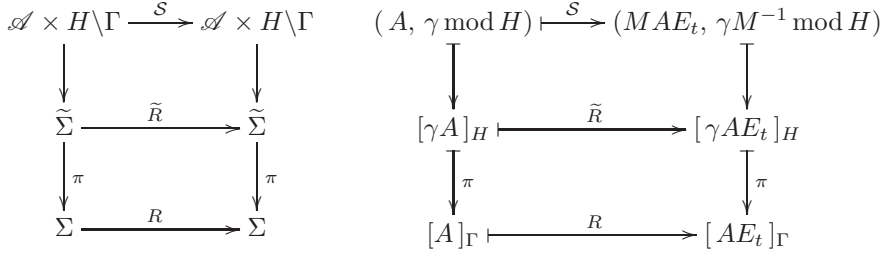

Since invariant subsets of a cross-section and of a flow correspond,  and since we know from  Hopf's theorem that the geodesic flow on a  finite volume hyperbolic surface is ergodic,  the following is evident.  
 
\begin{Cor}\label{c:skewErgodic}   The skew product transformation $\mathcal S$ is ergodic.
\end{Cor}

We note the following connection between flow cross-sections and fundamental domains.
\begin{Prop} \label{p:flowFund}  If $\mathscr A \subset G$ projects injectively to give a cross-section with return time $r$ for the geodesic flow on the surface uniformized by $\Gamma$, then a (measurable) fundamental domain for the action of $\Gamma$ on $G$ is given by the set
\[\{ A E_t\,|\,  A \in \mathscr A, 0 \le t \le r\}\] 
\end{Prop}

\section{Distribution of regular continued fraction approximants into cusps} 
Now we let $\Gamma = \text{PSL}(2, \mathbb Z)$ be the modular group,  and $H$ be any of its finite index subgroups, and begin with  the cross-section $\mathscr A$ for the geodesic flow on $\Gamma\backslash G$ given by Arnoux,  see Theorem~\ref{t:arnouxCross} above.      The first return map to this cross-section maps elements of the $(\sigma= -1)$-component to the $(\sigma = +1)$-component and {\em vice versa}.    This aspect of the map affects the presentation of the approximants $p_k/q_k$.

\begin{Lem}\label{l:prodMatricesModH}   Let $\mathscr A$ be the cross-section of the geodesic flow on the unit tangent bundle of the modular surface given in Theorem~\ref{t:arnouxCross}.   For $A_{-1}(x,y)\in \mathscr A$,   the second component of the $k^{\text{th}}$ composite of the skew product transformation $\mathcal S$ (defined in Theorem~\ref{t:skew}) with itself applied to $(A, I \;\text{mod}\, H)$ is 
\[
\emph{proj}_2(\,\mathcal S^k(A, I  \;\emph{mod}\, H)\,) = \begin{cases}  \begin{pmatrix}q_{k}&- q_{k-1}\\-p_k&p_{k-1}  \end{pmatrix}  \; \emph{mod}\; H&\text{if}\;   k \; \text{is even};\\
\\
                                                                              \begin{pmatrix}q_{k-1}&- q_{k}\\-p_{k-1}&p_{k}  \end{pmatrix}   \; \emph{mod}\; H&\text{if}\;  k \; \text{is odd}\;.\\
                                                          \end{cases}
\]                                                          
 \end{Lem}
 \begin{proof}   Since the initial value of $\sigma=-1$ and this value alternates in sign with each application of $\mathcal S$,   
 \[
 \begin{aligned}
 \text{proj}_2(\,\mathcal S^k(A, I  \;\text{mod}\, H)\,) &= M_{1}^{-1} M_{2}^{-1} \cdots M_{k}^{-1} \mod H\\
                                                                                        &= \begin{pmatrix}1&d_1\\0&1 \end{pmatrix}^{-1} \begin{pmatrix}1&0\\ d_2&1 \end{pmatrix}^{-1}\cdots M_{k}^{-1}\mod H\,.
 \end{aligned}
 \]
Upon referring to Equation~\eqref{e:convMatrixForm},  an elementary proof by induction verifies the statement. \end{proof}

 We need to address both the fact that the matrices above are of two types, and that the transpose  of $(p_k, q_k)$ is not a column of either of the two matrices displayed above.  We begin this with the following.    (Recall that by Definition~\ref{d:cuspClassDenote},  we can identify cusps with equivalence classes of $H\backslash \Gamma$ for the $\cdot p$ relation, where $p$ is a parabolic fixed point of $\Gamma$.)  In the following,  we denote the  indicator function for a set $E$ by $\mathbb 1_{\emph E}$, and we normalize the 
skew product measure to have total mass one.
 
\begin{Lem}\label{l:functionToEquivalenceClasses}      Fix a cusp $\kappa$ of $H$,  and let $\Psi$ be the function 

 \[ 
 \begin{aligned} \Psi: \;\;\;\mathscr A \times  H\backslash \Gamma &\to \mathbb R\\
      (\,A, \gamma \; \emph{mod}\;  H \,)&\mapsto  (    \mathbb 1_{\mathcal A_{-1}\times \{\kappa\}_{\cdot \infty}} + \mathbb 1_{\mathcal A_{+1}\times \{\kappa\}_{\cdot 0}}\,)/2\,.
  \end{aligned}
  \]    
Then, for almost every   $ (\,A, \gamma \; \emph{mod}\;  H \,)$   in  $\mathscr A \times H\backslash \Gamma$, the limit   
\[ \lim_{N\to \infty} \,\dfrac{1}{N}\sum_{k=1}^{N}\, \Psi\circ\mathcal S^k (\,A, \gamma \; \emph{mod}\;  H \,)\]
 is equal to the relative width of the cusp,  $w(\kappa)/[\Gamma:H]\,$.          
 \end{Lem}
 \begin{proof}   Since the function is a linear combination of indicator functions for sets of positive measure,  the ergodicity of  $\mathcal S$ and the Birkhoff Ergodic Theorem implies the result.  \
 \end{proof}

 At this point,  we can easily find that for almost every $x$,  the series of values $-q_k/p_k$ derived from its regular continued fraction approximants has  the appropriate frequency.  To obtain the Moeckel--Nakanishi result, we focus on the fact that  $\iota\cdot (-q_k/p_k) = p_k/q_k$.     That is, we now finish our proof of that result.  (When comparing with \eqref{e:moeckelLim}, recall that we now insist on $H$ denoting a projective group.) 
 
\begin{Cor}\label{c:switchWithIota}   For each cusp $\kappa$ of $H$ and for almost every $x \in \mathbb I$, we have 
 \[\lim_{N\to \infty}\; \dfrac{\#\{1\le k \le N\,|\; p_k/q_k\in \kappa\}}{N} = w(\kappa)/[\Gamma: H]\,.\]
 \end{Cor}
 \begin{proof}  The subset $\mathcal A_{-1} \times I\;\text{mod}\; H \subset \mathscr A \times H\backslash \Gamma$ has positive measure,  and hence the limit of Lemma~\ref{l:functionToEquivalenceClasses} holds for almost every element of it.     But,  this then implies that 
\[ \widetilde \Psi( \,A_{\sigma}(x,y), \gamma \; \text{mod}\;  H \,) = \begin{cases} [\, H  \gamma\,]_{\cdot \infty}&\text{if}\;   \sigma = -1\,;\\
\\
                                                                              [\,  H   \gamma\,]_{\cdot 0}&\text{if}\;   \sigma = +1\;.
                                                                               \end{cases}  
\] 
maps the corresponding sequence  $\mathcal S^k (\,A_{-1}(x,y), I \; \text{mod}\;  H \,)$ into the cusp $\kappa$ with 
 limiting frequency    equal to $w(\kappa)/[\Gamma:H]$.    From this it follows that for almost every $x \in \mathbb I$,  the sequence of  $-q_k/p_k$ has this limiting frequency.

 We now use the above with the role of $H$ replaced by $\iota H \iota$ --- we have that $-q_k/p_k$ has this limiting frequency in any cusp of $\iota H \iota$.    By Lemma~\ref{l:congGroups},   we find that the sequence $\iota\cdot (-q_k/p_k)$ has that same limiting frequency for corresponding cusps of $H$.   By this same Lemma,  the cusp widths are the same for the two groups.  But,  conjugate subgroups have the same index.     Therefore,  the result holds.  
 \end{proof}

\section{$\alpha$-continued fractions}
Nakada \cite{Nakada81} defined the following family of interval maps. 
For $\alpha \in (0,1]$, we let $\mathbb{I}_{\alpha} : = [\alpha-1,
\alpha)$ and define the map $T_{\alpha}:\, \mathbb{I}_{\alpha} \to \mathbb{I}_{\alpha}$ by
\[
T_{\alpha}(x) := \left| \frac{1}{x} \right| -  \left \lfloor\,  \left| \frac{1}{x} \right| + 1 -\alpha \right\rfloor,\ \mbox{for}\ x \neq 0\,; \ T_{\alpha}(0) :=0 \,.
\]
For $x \in \mathbb{I}_{\alpha}$, put
\[
\varepsilon(x) := \left\{\begin{array}{cl}1 & \mbox{if}\ x \ge 0\,, \\ -1 & \mbox{if}\ x<0\,,\end{array}\right. \quad \mbox{and} \quad d_{\alpha}(x) := \left\lfloor \left| \frac{1}{x} \right| + 1 - \alpha \right\rfloor\,,
\]
with $d_\alpha(0) = \infty$.

Furthermore, for $n \geq 1$, put
\[
\varepsilon_n = \varepsilon_{\alpha,n}(x) := \varepsilon(T^{n-1}_\alpha (x)) \quad \mbox{and}\quad d_n = d_{\alpha,n}(x) := d_\alpha(T^{n-1}_\alpha (x)).
\]
This yields the \emph{$\alpha$-continued fraction}  expansion of
$x\in\mathbb R$\,:
\[
x = d_0+ \dfrac{\varepsilon_1}{d_1 + \dfrac{\varepsilon_2}{d_2 +
\cdots}} 
\,,
\]
where $d_0\in\mathbb Z$ is such that $x-d_0\in \mathbb{I}_{\alpha}$.
(Standard convergence arguments again justify equality of $x$ and its
expansion.) These include the regular continued fractions, given by
$\alpha = 1$ and the nearest integer continued fractions, given by
$\alpha = 1/2$.

We define the $\alpha$-approximants to $x$ so that the
recurrence relations
\begin{eqnarray*}
p_{-1}=1; \, p_0=0; & p_n= d_n p_{n-1}+\varepsilon_n p_{n-2},\,
n\geq 1\\
q_{-1}=0; \, q_0=1; & q_n= d_n q_{n-1}+\varepsilon_n q_{n-2},\,
n\geq 1,
\end{eqnarray*}
hold. One easily checks that $\left|
p_{n-1}q_{n} \, - \, q_{n-1}p_n \right| =1$.

We show in this section that the distribution of the $\alpha$-continued fraction approximants $p_k/q_k$ is the same as that for the regular continued fractions. 
\begin{Thm}\label{t:alphaMoeckel}    Suppose that $H$ is a finite index subgroup $H \subset \emph{PSL}(2, \mathbb Z)$.    Then,  for each $\alpha\in (0,1]$ and for almost every real $x$,  the $\alpha$-continued fraction approximants of $x$ are distributed in the cusps of $H$ according to the relative cusp widths.  
\end{Thm} 

\bigskip 
The standard number theoretic planar map associated to these continued fractions is defined by
\[
\mathcal{T}_\alpha(x,y) := \bigg(T_\alpha(x), \frac{1}{d_{\alpha}(x)+ \varepsilon(x)\,y}\bigg)\,, \quad (x,y) \in \Omega_{\alpha}\,,
\]
see \cite{KraaikampSSteiner12} for a description of $\Omega_{\alpha}\subset \mathbb I_{\alpha}\times [0,1]$.
Here also, $\mu$ as given in~\eqref{e:mu} is an invariant measure.
 
Let $\mu_\alpha$~be the probability measure given by normalizing $\mu$ on~$\Omega_\alpha$, and $\nu_\alpha$~the probability measure normalized from the marginal measure obtained by integrating $\mu_\alpha$ over the fibers~$\{x\} \times \{y \mid (x,y) \in \Omega_\alpha\}$, $\mathscr{B}_\alpha$~the Borel $\sigma$-algebra of~$\mathbb{I}_\alpha$, and $\mathscr{B}_\alpha'$ the Borel $\sigma$-algebra of~$\Omega_\alpha$.
Kraaikamp-Schmidt-Steiner \cite{KraaikampSSteiner12}  showed that  $(\Omega_\alpha, \mathcal{T}_\alpha, \mathscr{B}_\alpha', \mu_\alpha)$ is a natural extension of $(\mathbb{I}_\alpha, T_\alpha, \mathscr{B}_\alpha, \nu_\alpha)$.

Analogously to the regular continued fraction case of $\alpha=1$, let 
\[ \mathcal A_{\alpha,  -1} = \{ A_{-1}(x,y) \,     \vert \, (x, y) \in \mathcal Z^{-1}(\Omega_{\alpha})\}\;\mbox{and} \;  \mathcal A_{+1} = \{ \; A_{+1}(x,y) \, \vert \, (x, y) \in\mathcal Z^{-1}(\Omega{\alpha})\}\,,   \]
where $\mathcal Z$ is given in~\eqref{e:zMap};
and,  let $\mathscr A_{\alpha} = \mathcal A_{\alpha, -1}  \cup \mathcal A_{\alpha, +1}$, considered as elements of $G$. 
      
Arnoux-Schmidt \cite{ArnouxSchmidt12} show that 
\begin{Thm}\label{t:alphaCross} [Arnoux-Schmidt]   The set $\mathscr A_{\alpha} \subset G$ projects injectively (up to measure zero) to a cross-section for the geodesic flow on the unit tangent bundle of the modular surface.      Furthermore,   the first return map to this cross-section by the geodesic flow is given by the projection of 
\[ 
A_{\sigma}(x,y)  \mapsto  M A_{\sigma}(x,y) E_t \,,  
\]
with $t = -2 \log |x|$  and $M$ equaling 
\[ 
 \begin{pmatrix}0&1\\-1&d \end{pmatrix} ,   \begin{pmatrix}1&d\\0&1 \end{pmatrix},   \begin{pmatrix}d &-1\\1&0 \end{pmatrix} ,   \begin{pmatrix}1&0\\ d&1 \end{pmatrix}\,,\]
where $d = d_{\alpha}(x)$,  as $(\sigma, \varepsilon(x)) = (-1,-1), (-1, +1), (+1, -1), (+1, +1)$,  respectively.  Moreover,  the dynamical system defined by this first return map has the map $T_{\alpha}$ as a factor.  
\end{Thm}

 Note that when $\varepsilon(x) = +1$,    $M$ is of the same shape as in the regular continued fraction case given by Arnoux, see Theorem~\ref{t:arnouxCross}.  In all cases, $\mathscr A_{\alpha}$ has two components, indexed by  $\sigma = \pm 1$.   
 In  general,  the map sends an element with $\varepsilon = +1$ in the component indexed by $\sigma$ to the component indexed by $-\sigma$, whereas elements with $\varepsilon = -1$   are mapped to within the same component.    This is key to showing that the analog of Lemma~\ref{l:prodMatricesModH} holds with only two types of matrices, instead of four as the previous result might suggest. 
 
\begin{Lem}\label{l:prodAlphaMatricesModH}   Let $\mathscr A_{\alpha}$ be the cross-section of the geodesic flow on the unit tangent bundle of the modular surface described in Theorem~\ref{t:alphaCross}.   Suppose that $A_{-1}(x,y)\in \mathscr A_{\alpha}$,  then the second component of the $k^{\text{th}}$ composite of the skew product transformation $\mathcal S$ with itself applied to $(A, I \;\text{mod}\, H)$ is 
\[
\emph{proj}_2(\,\mathcal S^k(A, I  \;\emph{mod}\, H)\,) = \begin{cases}  \begin{pmatrix}q_{k}&- q_{k-1}\\-p_k&p_{k-1}  \end{pmatrix}  \; \emph{mod}\; H&\text{if}\;\,  \emph{proj}_1(\,\mathcal S^k(A, I  \;\emph{mod}\, H)\,)  \in \mathcal A_{\alpha, -1}\;;\\
\\
                                                                              \begin{pmatrix}q_{k-1}&- q_{k}\\-p_{k-1}&p_{k}  \end{pmatrix}   \; \emph{mod}\; H&\text{otherwise}.\\
                                                          \end{cases}
\]                                                          
 \end{Lem}

Note that the condition defining the two cases above becomes that of Lemma~\ref{l:prodMatricesModH} upon restricting to the case of all $\varepsilon_i(x) = +1$.

 \begin{proof}     For $A = A_{-1}(x,y)\in \mathcal A_{\alpha, -1}$, we have  $\text{proj}_1(\,\mathcal S^k(A, I  \;\text{mod}\, H)\,)  \in \mathcal A_{\alpha, -1}$ if and only if  the number of $\varepsilon_i(x) = +1$ for $1\le i \le k$ is even.     The direct calculation verifying Lemma~\ref{l:prodMatricesModH} shows that each application of $\mathcal S$ with  $\varepsilon_i(x) = +1$ changes between the two shapes of matrices shown.   Another direct calculation shows that each application of $\mathcal S$  with $\varepsilon_i(x) = -1$ preserves the respective shapes.    The result follows.
\end{proof}
 
  It is now immediate that the analogous statements to Lemma~\ref{l:functionToEquivalenceClasses}  and Corollary~\ref{c:switchWithIota} hold, and we indeed conclude that the limiting distribution into cusps of the $\alpha$-approximants is the same as that for the regular continued fraction approximants.

\section{Rosen continued fractions}
Let $\lambda = \lambda_m = 2 \cos \frac{\pi}{m}$, then $\Gamma_m$, the Hecke group of index $m$ is the  single-cusped Fuchsian group generated by the elements $\iota$ given above in \eqref{e:iota}, and 
\[
\gamma_m = \begin{pmatrix}
	1   &  \lambda_m\\
 	0   &  1
\end{pmatrix}\,.
\]
The group $\Gamma_3$ is the modular group $\text{PSL}(2, \mathbb Z)$.   For $m \notin \{3,4,6\}$,  the group $\Gamma_m$ is non-arithmetic, see \cite{Leutbecher67}; and, thus in particular has no finite index subgroup that is $\text{PSL}(2, \mathbb R)$-conjugate to a finite index subgroup of $\text{PSL}(2, \mathbb Z)\,$.

In 1954, D. ~Rosen  \cite{Rosen54} defined 
an infinite family of continued fraction algorithms to aid in the study of the Hecke groups.   The Rosen continued fractions and variants   have been of recent interest, leading to results especially about their dynamical and arithmetical properties, see \cite{BurtonKraaikampSchmidt00}, \cite{Nakada10}, \cite{DajaniKraaikampSteiner09}; as well on their applications to the study of geodesics on related hyperbolic surfaces, see \cite{SchmidtSheingorn95}, \cite{BogomolnySchmit04}, \cite{MayerStroemberg08};   and to Teichm\"uller geodesics arising from (Veech) translation surfaces, see  \cite{ArnouxHubert00}, \cite{SmillieUlcigrai10},\cite{SmillieUlcigrai11}  and \cite{ArnouxSchmidt09}.   
Several basic questions remain open, including that of arithmetically characterizing the real numbers having a finite Rosen continued fraction expansion, 
see \cite{Leutbecher67}, \cite{HansonMerbergTowseYudovina08} and  \cite{ArnouxSchmidt09}.      

We show in this section that the Rosen continued fraction approximants $p_k/q_k$ are distributed analogously to those  for the regular continued fractions.    (We note that \cite{Nakada10} gives the following result in the setting of the distribution into cusps of the principal congruence subgroups of the Hecke groups,  using an approach directly related to Moeckel's.)

\begin{Thm}\label{t:rosenMoeckel}    Fix an integer $m\ge 3$, and suppose that $H$ is a finite index subgroup of the Hecke group of index $m$, $H \subset \Gamma_m$.      Then,  for each for almost every real $x$,  the Rosen $\lambda_m$-continued fraction approximants of $x$ are distributed in the cusps of $H$ according to the relative (to $\Gamma_m)$ cusp widths.  
\end{Thm} 
\bigskip

Let
${\mathbb I}_m = [-\lambda/2, \lambda/2\,)$ for $m
\ge 3$. For a fixed integer $m \ge 3$, the Rosen continued fraction
map is defined by
\[
T_m(x) = \begin{cases} \left| \frac{1}{x} \right| \, - \,
\lambda \lfloor \left| \,\frac{1}{\lambda x} \right| + \frac{1}{2}
\rfloor & x \ne 0; \\
\\
  0   &  x = 0
\end{cases}
\]
for $x \in {\mathbb I}_m$; here and below, we  omit the index
``$m$" whenever it is clear from context.  For $n \geq 1$, we define 
\[
\varepsilon_n(x) = \varepsilon (T_{m}^{n-1}x) \qquad \mbox{and} \qquad  
r_n(x) = r(T_{m}^{n-1}x)
\]
with
\[
\varepsilon (y) = \mbox{sgn}(y) \qquad \mbox{and} \qquad  
r(y) = \left\lfloor \, \left| \frac{1}{\lambda y} \right| + \frac{1}{2} 
\right\rfloor .
\]
Then, as Rosen showed in \cite{Rosen54},  the Rosen continued fraction expansion of $x$ 
is given by
\[
[\, \varepsilon_1(x):r_1(x),\,
\varepsilon_2(x):r_2(x),\ldots ,\, \varepsilon_n(x):r_n(x),\ldots ] : = \dfrac{\varepsilon_1}{r_1 \lambda + \dfrac{\varepsilon_2}{r_2 \lambda + \cdots}}\,.\]
  As usual we define the {\em approximants} $p_n/q_n$
of $x\in {\mathbb I}_m$ by
\[
\begin{pmatrix}
p_{-1}  &  p_0  \\
q_{-1}  &  q_0
\end{pmatrix}
\, = \,
\begin{pmatrix}
1  &  0  \\
0  &  1
\end{pmatrix}
\]
and
\[
\begin{pmatrix}
p_{n-1}  &  p_{n}  \\
q_{n-1}  &  q_{n}
\end{pmatrix}
\, = \,
\begin{pmatrix}
0 &  \varepsilon_{1}  \\
1  &  \lambda r_{1}
\end{pmatrix}
\begin{pmatrix}
0 &  \varepsilon_{2}  \\
1  &  \lambda r_{2}
\end{pmatrix}
\cdots
\begin{pmatrix}
0 &  \varepsilon_{n}  \\
1  &  \lambda r_{n}
\end{pmatrix}
\]
for $n \ge 1$. From this definition it is immediate that $\left|
p_{n-1}q_{n} \, - \, q_{n-1}p_n \right| =1$, and that the following 
well-known recurrence relations hold:
\begin{eqnarray*}
p_{-1}=1; \, p_0=0; & p_n=\lambda r_n p_{n-1}+\varepsilon_n p_{n-2},\,
n\geq 1\\
q_{-1}=0; \, q_0=1; & q_n=\lambda r_n q_{n-1}+\varepsilon_n q_{n-2},\,
n\geq 1.
\end{eqnarray*}

\bigskip 

\begin{proof}[Sketch of proof of Theorem~\ref{t:rosenMoeckel}]

  For each $m$,  Burton-Kraaikamp-Schmidt \cite{BurtonKraaikampSchmidt00} determined a planar natural extension on a region $\Omega_m$ with the measure $\mu$ as above.  Using this, Arnoux-Schmidt \cite{ArnouxSchmidt12} determine an  $\mathscr A_m \subset G = \text{PSL}(2, \mathbb R)$ such that $\mathscr A_m$ projects to a cross-section of the geodesic flow on $\Gamma_m\backslash G$.  The first return map to this cross-section is a double cover of the natural extension.     Furthermore,  for each index $m$ an analog of Theorem~\ref{t:alphaCross} holds as \cite{ArnouxSchmidt12} show, where again when  $\varepsilon_i(x) = +1$ there is a change of component,  and when $\varepsilon_i(x) = -1$ there is not;   the matrices giving the return to the cross-section here are
\[ 
 \begin{pmatrix}0&1\\-1&d \lambda_m\end{pmatrix},\;    \begin{pmatrix}1&d \lambda_m\\0&1 \end{pmatrix},   \; \begin{pmatrix}d\lambda_m &-1\\1&0 \end{pmatrix}, \;   \begin{pmatrix}1&0\\ d\lambda_m&1 \end{pmatrix}\,,\]
where $d = d_{\alpha}(x)$,  as $(\sigma, \varepsilon(x)) = (-1,-1), (-1, +1), (+1, -1), (+1, +1)$,  respectively.

 Therefore, we easily find that the analog of Lemma~\ref{l:prodAlphaMatricesModH} holds.    Now, for each $m$ we have both that $\gamma_m$ fixes  $\infty$, and $\iota \in \Gamma_m$,  this allows the proof of Corollary~\ref{c:switchWithIota} to apply.   From this,  one easily completes the proof of Theorem~\ref{t:rosenMoeckel}.
\end{proof}  
 
\bigskip 
As an example of an implication of this, we consider the case of $q=5$.  Note that $\lambda_5$ is the golden ration, $(1+\sqrt{5})/2$.   
\begin{Eg}\label{c:rosen5Even}     Almost every real number  has  Rosen $\lambda_5$-continued fraction approximants  with equal asymptotic frequency of the five types:   {\tt odd/even}, {\tt even/odd},   $1/1$, $1/\lambda_5$, $1/(\lambda_5+ 1)$.

The proof of this relies on checking the data related to the principal congruence subgroup of $G_5$  for the $\mathbb Z[\lambda_5]$ ideal generated by $2$ --- this (normal) subgroup has five cusps, each of width two.     We have chosen cusp representatives in the $(\gamma_5 \iota)$-orbit of $\infty$.

We briefly sketch a non-geometric manner to see that there should indeed be five cusps for this principal congruence subgroup.   
 Note that $\mathbb Z[\lambda_5]$  is the full ring of integers of its quotient field,  and that the rational prime ideal $\langle 2\rangle$ remains prime (is inert to this number field extension).   Thus, each element of a  pair $(p_k,q_k)$  reduces modulo $\langle 2 \rangle$ to some element of the quotient field $\mathbb Z[\lambda_5]/\langle 2\rangle$, a finite field with four elements.   However,  each pair is relatively prime,  thus there are obviously at most twelve combinations possible.   Now,  if either of $p_k, q_k$ is a multiple of $2$,  then multiplying each by an inverse modulo $2$ of the other reduces the quotient to the form of  one of the cusps  {\tt odd/even}, {\tt even/odd}.    For the six remaining possibilities,  one multiplies through by a multiplicative inverse of $p_k$ to see membership in one of the three remaining cusps.   
 \end{Eg} 

 \section{Further remarks}  
 
 We note that Theorem~\ref{t:skew} remains true upon replacing the geodesic flow by the flow defined by any one-parameter subgroup of $G$; if this new flow is represented by the right multiplication by matrices of the form $F_t$,  then 
Lemma~\ref{l:returnPair} and the theorem simply require the replacement of the $E_t$ by these $F_t$.     Of course the ergodicity of the skew product transformation announced in Corollary~\ref{c:skewErgodic}, depends on the flow at hand.
For the (stable and unstable) horocycle flows,  results of Hedlund \cite{Hedlund36} and later of Dani and Smillie \cite{DaniSmillie84} can be invoked,  so ergodicity (and even  essential unique ergodicity;   that is, unique ergodicity after the suppression of weights on the collection of periodic orbits) of the skew product again holds.

\bibliographystyle{amsalpha}
\bibliography{skewMoeckelBib}

\end{document}